\documentclass[11pt]{amsart}
\usepackage{amsmath,amsthm}
\usepackage{amsfonts}
\usepackage{enumerate}
\usepackage[psamsfonts]{amssymb}
\usepackage{hyperref}
\usepackage{mathtools}
\usepackage{color}
\usepackage{geometry}

\newcommand{\field}[1]{{\mathbb{#1}}}
\newcommand{\C}{\field{C}}
\newcommand{\N}{\field{N}}

\newcommand{\R}{\field{R}}
\newcommand{\Z}{\field{Z}}
\newcommand{\CC}{\mathcal{{C}}}
\newcommand{\dd}{\,\mathrm{d}}
\newcommand{\intpart}[1]{\left\lfloor#1\right\rfloor}
\newcommand{\eps}{\varepsilon}

\newcommand{\ph}{{\vphantom{*}}}                         % a null vertical space: using T^*_\alpha T^\ph_\alpha instead of T^*_\alpha T_\alpha we get subscript at the same level.
\newcommand{\Gbinom}[2]{\genfrac{[}{]}{0pt}{}{#1}{#2}}   % Gaussian binomial symbol

\DeclareMathSymbol{\sm}{\mathbin}{AMSa}{"39}             % the short minus used for limits to the left.

\allowdisplaybreaks[4] %in AMSLaTeX consente lo spezzamento su pi\`{u} pagine degli ambienti align e simili

\theoremstyle{remark}

\newtheorem*{remark*}{Remark}
\newtheorem*{example*}{Example}
\newtheorem*{conjecture*}{Conjecture}

\theoremstyle{plain}
\newtheorem{proposition}{Proposition}
\newtheorem{theorem}{Theorem}
\newtheorem{lemma}{Lemma}
\newtheorem{corollary}{Corollary}

\begin{document}
\title[A one parameter family of operators]{A one parameter family of Volterra-type operators}

%------------------------ AUTHORS -------------------------
\author[F.~Battistoni]{Francesco Battistoni*}
\address[F.~Battistoni]{Dipartimento di Matematica per le Scienze Economiche, Finanziarie ed Attuariali\\
         Universit\`{a} Cattolica\\
         via Necchi 9\\
         20123 Milano\\
         Italy}
\email{francesco.battistoni@unicatt.it}
\thanks{*ORCID: \url{https://orcid.org/0000-0003-2119-1881}}

\author[G.~Molteni]{Giuseppe Molteni\textsuperscript{\textdagger}}
\address[G.~Molteni]{Dipartimento di Matematica\\
         Universit\`{a} di Milano\\
         via Saldini 50\\
         20133 Milano\\
         Italy}
\email{giuseppe.molteni1@unimi.it}
\thanks{\textsuperscript{\textdagger}Corresponding author; ORCID: \url{https://orcid.org/0000-0003-3323-4383}}


% 45D05 = Volterra integral equations
% 45P05 = Integral operators
% 47A10 = spectrum, resolvent
% 47A30 = Norms (inequalities, more than one norm, etc.) of linear operators
% 47B38 = Linear operators on function spaces (general)
% 47G10 = Integral operators

\begin{abstract}
For every $\alpha \in (0,+\infty)$ and $p,q \in (1,+\infty)$ let $T_\alpha$ be the operator $L^p[0,1]\to
L^q[0,1]$ defined via the equality $(T_\alpha f)(x) := \int_0^{x^\alpha} f(y) \dd y$. We study the norms
of $T^\ph_\alpha$ for every $p$, $q$. In the case $p=q$ we further study its spectrum, point spectrum,
eigenfunctions, and the norms of its iterates. Moreover, for the case $p=q=2$ we determine the point
spectrum and eigenfunctions for $T^*_\alpha T^\ph_\alpha$, where $T^*_\alpha$ is the adjoint operator.
\end{abstract}

\maketitle

\section{Introduction}\label{sec:1A}
Let $p,q\in (1,+\infty)$ and let $p'$, $q'$ denote the conjugated exponent of $p$ and $q$, respectively.
Pick any real $\alpha >0$ and let $T_\alpha\colon L^p[0,1] \to L^q[0,1]$ be defined as
\[
 (T_\alpha f)(x) := \int_0^{x^\alpha} f(y) \dd y.
\]
This operator can be considered as an interpolation depending on $\alpha$ of the projector on the
subspace of constants $T_0$, and the null operator $T_\infty$, passing through classical Volterra
operator $V = T_1$.

In the spirit of earlier introductions by Tonelli~\cite{Tonelli} and Tikhonov~\cite{Tikhonov} of the
class of Volterra operators, an operator $T$ should be called ``Volterra type'' when $(T f)(x)$ depends
on the values of $f$ only in its ``past'', i.e. in $[0,x]$. In effect, more general classes have also
been considered and termed ``Volterra'', see for example~\cite{DrakhlinLitsyn} and the extensive
bibliography cited therein, but usually modern presentations of this class of operators still retain in
some form that property, plus extra hypotheses about their domain (Hilbert spaces, Banach space, general
$L^p$ spaces) and about the regularity of the involved kernel associated with the operator.
Just to cite the most relevant for this work, this is the case of Barnes~\cite{Barnes}, where general
properties of the spectrum for Volterra operators on $L^p$ spaces are studied,
Eveson~\cite{Eveson,Eveson2}, where the norm of iterations of Volterra operators with convolutive kernels
are studied on $L^2[0,1]$ and in $L^p[0,1]$ cases respectively,
% norm of T^n on L^2, where T f(x) = \int_0^x k(t-s)f(s) ds    (convolution kernel)
or even Adell and Gallardo--Gutierrez~\cite{AdellGallardo-Gutierrez}, where the norm of the
Liouville--Riemann fractional integration operators
%\frac{1}{Gamma(s)}\int_0^x (x-t)^{s-1} f(t) dt
% for s diverging (when s is an integer it coincides with the s-th iterate of the classical Volterra operator).
%
are studied for large values of the parameter, again in $L^p$ spaces.\\
Under this point of view and despite the name we have adopted to denote them, operators $T_\alpha$ with
$\alpha < 1$ represent a deviation to this tradition, since for them $(T_\alpha f)(x)$ depends on the
values of $f$ on the strictly larger range $[0,x^\alpha]$. In fact, these operators show more varied and
somewhat unexpected behaviours; for example they are not quasi-nilpotent, see Theorem~\ref{th:3A}.
\smallskip\\
Every $T_\alpha$ is a compact operator, as one can readily deduce from the Kolmogorov--Riesz Theorem and
the fact that
\[
 |(T_\alpha f)(x) - (T_\alpha f)(y)| \leq |x^\alpha - y^\alpha|^{1/p'} \|f\|_p
\]
for every $x,y\in[0,1]$ and every $f$ in $L^p[0,1]$.
Moreover, the adjoint $T_\alpha^\ast\colon L^{q'}[0,1] \to L^{p'}[0,1]$ is
\[
T_\alpha^\ast f = \int_{x^{1/\alpha}}^1 f(u) \dd u,
\]
so that we have the equality
\begin{equation}\label{eq:1A}
(T_0 - T_\alpha)^* = T_{1/\alpha},
\end{equation}
where now $T_{1/\alpha}$ is considered as a map $L^{q'}[0,1] \to L^{p'}[0,1]$.\\
Equation~\eqref{eq:1A} shows that there are relations connecting operators $T_\alpha$, $T_{1/\alpha}$,
$T^*_\alpha$ and $T^*_{1/\alpha}$, so that it is a good idea to consider the full family
$\{T_\alpha\}_{\alpha>0}$ as a whole, with the classical Volterra operator $V=T_1$ playing in some sense
a pivotal role under the correspondence $\alpha \leftrightarrow 1/\alpha$.
\medskip

In Section~\ref{sec:2A} we study the norm of $T_\alpha$. In Section~\ref{sec:3A} we determine its
spectrum $\sigma$, point spectrum $\sigma_0$ and eigenfunctions. Section~\ref{sec:4A} contains a result
about the behaviour of the norm of iterates $T_\alpha^n$ of high order. Finally, Section~\ref{sec:5A} is
devoted to the special case $p=q=2$, where we can explicitly describe the spectrum and the eigenvalues of
$T^*_\alpha T^\ph_\alpha$, and deduce the exact value of the norm of $T_\alpha$.
\bigskip\\ %
{\bf Notations:} %
We frequently use Landau symbols $f(x) = O(g(x))$, $f(x) = o(g(x))$ and $f(x) \asymp g(x)$ as $x\to
x_0\in \R\cup\{\pm\infty\}$ with the meaning that $|f(x)/g(x)|$ stays bounded, $|f(x)/g(x)|$ goes to $0$
and both $|f(x)/g(x)|$ and $|g(x)/f(x)|$ stay bounded, respectively. The presence of a subscript in any
of such symbols means that the symbol is not uniform in the parameter appearing in the subscript.
\bigskip\\ %
{\bf Acknowledgments:} %
%We thank the anonymous referee for their careful reading.
The first author is member of the GNAMPA research group, the second author of the GNSAGA research group.

\section{Norm}\label{sec:2A}
Each operator $T_\alpha$ is positive, so that its norm can be computed using nonnegative test functions.
For every such a function $f$ and every fixed $x$ the map $\alpha \to (T_\alpha f)(x)$ is decreasing so
that also the map $\alpha \to \|T_\alpha\|_{p,q}$ decreases.
% NOTICE:
% for $a < b$ and $f \geq 0$, $x^a\geq x^b$, hence $(T_a f)(x) \geq (T_b f)(x) \geq 0$, hence
% $\|T_a f \|_q \geq \| T_b f\|_q$, i.e. \|T_a\|_{p,q} \geq \|T_b\|_{p,q}.
% MOREOVER,
% in case $f'$ exists and is positive, one has
% $\partial_\alpha (T_\alpha f)(x) = \alpha x^\alpha \log x f(x^\alpha)$
% and
% $\partial^2_\alpha (T_\alpha f)(x) = \alpha x^\alpha \log^2 x f(x^\alpha) + \alpha x^{2\alpha} \log^2 x f'(x^\alpha)$
% which is positive, so that in this case $\alpha \to (T_\alpha f)(x)$ is convex, giving:
% $(T_{\lambda a + (1-\lambda)b} f)(x) \leq \lambda (T_a f)(x) + (1-\lambda) (T_b f)(x)$
% so that
% $\| T_{\lambda a + (1-\lambda)b} f\|_q \leq \|\lambda (T_a f) + (1-\lambda) (T_b f)\|_q
% \leq \lambda \|T_a f\|_q + (1-\lambda) \|T_b f\|_q$ giving
% $\|T_{\lambda a + (1-\lambda)b}\|_{p,q} \leq \lambda \|T_a\|_{p,q} + (1-\lambda)\|T_b\|_{p,q}$
%
The following result proves that the maps $\alpha \to T_\alpha$ and $\alpha \to \|T_\alpha\|_{p,q}$ are
H\"{o}lder continuous in $[0,+\infty)$ of order $1/p'$, at least.
\begin{theorem}\label{th:1A}
Let $p,q \in (1,+\infty)$ and let $p'$ be the conjugated exponent of $p$. Let $\alpha,\beta \geq 0$. Then
\[
\big|\|T_\alpha\|_{p,q} - \|T_\beta\|_{p,q}\big|
\leq \|T_\alpha - T_\beta\|_{p,q}
\leq |\alpha-\beta|^{1/p'} \Gamma(\tfrac{q}{p'} + 1)^{1/q}.
\]
\end{theorem}
\begin{proof}
Let $f\in L^p[0,1]$. H\"{o}lder's inequality gives
\[
|(T_\alpha f)(x) - (T_\beta f)(x)|
= \Big|\int_{x^\alpha}^{x^\beta} f(y) \dd y\Big|
\leq \|f\|_p |x^\alpha-x^\beta|^{1/p'}
\qquad \forall x\in [0,1],
\]
so that
\[
\|T_\alpha - T_\beta\|_{p,q}
\leq \Big[\int_0^1 |x^\alpha-x^\beta|^{q/p'}\dd x \Big]^{1/q}.
\]
By the mean value Theorem, there exists $\eta$ between $\alpha$ and $\beta$ and depending on $x$ such
that
\[
|x^\alpha - x^\beta| = |\alpha - \beta| x^\eta |\log x|,
\]
in particular
\[
|x^\alpha - x^\beta| \leq |\alpha - \beta| |\log x|
\qquad \forall x\in [0,1]
\]
so that
\begin{align*}
\|T_\alpha - T_\beta\|_{p,q}
\leq |\alpha-\beta|^{1/p'} \Big[\int_0^1 |\log x|^{q/p'}\dd x \Big]^{1/q}
= |\alpha-\beta|^{1/p'} \Gamma(\tfrac{q}{p'}+1)^{1/q}.
\end{align*}
\end{proof}
The following theorem provides lower/upper bounds for the norm and some hints about its behaviour when
$\alpha$ tends to $\infty$ and to $0$.
\begin{theorem}\label{th:2A}
Let $p,q \in (1,+\infty)$ and let $p'$ and $q'$ be the conjugated exponent of $p$ and $q$, respectively.
Then
\begin{equation}\label{eq:2A}
(\alpha q + 1)^{-1/q}
\leq
\|T_\alpha\|_{p,q}
\leq
\min\Big\{\big(\alpha\tfrac{q}{p'}+1\big)^{-1/q} , \big[\alpha B\big(\tfrac{p'}{q}+1,\alpha\big)\big]^{1/p'}\Big\},
\end{equation}
where $B(a,b):=\int_0^1 x^{a-1}(1-x)^{b-1} \dd x$, the Euler Beta function. In particular
\begin{alignat}{2}
\|T_\alpha\|_{p,q}       &\asymp \alpha^{-1/q}
&\qquad&\text{as}\
\alpha \to \infty,                              \label{eq:3A}\\
\|T_\alpha - T_0\|_{p,q} &\asymp \alpha^{1/p'}
&\qquad&\text{as}\
\alpha \to 0,                                   \label{eq:4A}\\
|\|T_\alpha\|_{p,q} - 1| &\asymp \alpha
&\qquad&\text{as}\
\alpha \to 0.                                   \label{eq:5A}
\end{alignat}
\end{theorem}
\noindent %
The result in~\eqref{eq:4A} refines the case $\beta = 0$ in Theorem~\ref{th:1A}, since now also the
lower bound is proved.
Moreover, it is interesting to notice that in the limit $\alpha\to 0$ the operator $T_\alpha$ tends to
$T_0$ as $\alpha^{1/p'}$ by~\eqref{eq:4A}, while $\|T_\alpha\|_{p,q}$ tends to $\|T_0\|_{p,q}=1$ as
$\alpha$ by~\eqref{eq:5A}, so that the norm converges far better than the operator.
\begin{proof}
The lower bound comes quickly by comparing $\|T_\alpha f\|_q$ and $\|f\|_p$ for $f=1$.\\
The first upper bound comes from H\"{o}lder's inequality: let $f\in L^p[0,1]$, then
\begin{align*}
\|T_{\alpha}f\|_q^q
&%=    \int_0^1\Big|\int_0^{x^{\alpha}}f(y)dy\Big|^q\dd x
 \leq \int_0^1\Big(\int_0^{x^{\alpha}}|f(y)|dy\Big)^q\dd x\\
&\leq \int_0^1\Big[\Big(\int_0^{x^{\alpha}}|f(y)|^p \dd y\Big)^{1/p}\Big(\int_0^{x^{\alpha}}1^{p'} \dd y\Big)^{1/p'}\Big]^q\dd x\\
&\leq \|f\|_p^q \int_0^1 x^{\alpha q/p'}\dd x
=     \|f\|_p^q \big(\alpha\tfrac{q}{p'}+1\big)^{-1}.
\end{align*}
Also the second upper bound comes from H\"{o}lder's inequality, but used for the operator $T^*_\alpha\colon$
$L^{q'}[0,1] \to L^{p'}[0,1]$, via the equality $\|T_{\alpha}\|_{p,q} = \|T_{\alpha}^*\|_{q',p'}$.
In fact, let $f\in L^{q'}[0,1]$. Then
\begin{align*}
\|T_{\alpha}^* f\|_{p'}^{p'}
&\leq \int_0^1\Big(\int_{x^{1/\alpha}}^1 |f(y)|\dd y \Big)^{p'}\dd x
 \leq \int_0^1\Big[\Big(\int_{x^{1/\alpha}}^1|f(y)|^{q'}\Big)^{1/q'}\Big(\int_{x^{1/\alpha}}^1 1^{q}\Big)^{1/q}\Big]^{p'}\dd x\\
&\leq \|f\|_{q'}^{p'}\int_0^1(1-x^{1/\alpha})^{p'/q}\dd x
 =    \|f\|_{q'}^{p'}\alpha\int_0^1 (1-z)^{p'/q}z^{\alpha-1}\dd z\\
&=    \|f\|_{q'}^{p'}\alpha B\big(\tfrac{p'}{q}+1,\alpha\big).
\end{align*}
The comparison of the lower bound and the first upper bound in~\eqref{eq:2A} gives~\eqref{eq:3A}
and~\eqref{eq:5A}.\\
To prove~\eqref{eq:4A} we take advantage of the relation~\eqref{eq:1A}, so that
\begin{align*}
\|T_\alpha-T_0\|_{p,q}
= \|(T_\alpha-T_0)^*\|_{q',p'}
= \|T_{1/\alpha}\|_{q',p'}
\end{align*}
and the claim follows from~\eqref{eq:3A}.
\end{proof}
\noindent %
It is immediate to verify that the two upper bounds in Theorem~\ref{th:2A} coincide when $q=p'$, for
every $\alpha$. The following proposition shows that this is the unique case where this happens, and
describes explicitly when one of the two upper bounds is more convenient than the other.
\begin{proposition}\label{prop:1A}
Let $p,q \in (1,+\infty)$ and let $p'$ be the conjugate exponent of $p$. Then
\[
\begin{array}{ll}
\big[\alpha B\big(\alpha, 1+\tfrac{p'}{q}\big)\big]^{1/p'} \geq \big(\alpha\tfrac{q}{p'}+1\big)^{-1/q} & \text{ if } q \geq p',\\
\big[\alpha B\big(\alpha, 1+\tfrac{p'}{q}\big)\big]^{1/p'} \leq \big(\alpha\tfrac{q}{p'}+1\big)^{-1/q} & \text{ if } q \leq p',
\end{array}
\]
with equality if and only if $q=p'$.
\end{proposition}
\begin{proof}
In terms of $u := p'/q$ the problem is equivalent to studying the sign of the function
\begin{align*}
F(\alpha,u)
:&= \log\big[\alpha B(\alpha, 1+u)\big] + u\log\Big(1+\frac{\alpha}{u}\Big)\\
 &= \log\Gamma(\alpha+1)+\log\Gamma(1+u) - \log\Gamma(1+u+\alpha) + u\log\Big(1+\frac{\alpha}{u}\Big),
\end{align*}
where the last equality comes from the representation of the Beta function as product of gammas
(see~\cite[Th.~1.1.4]{specialFunctions}).
Since $F(0,u)=0$ for every $u>0$, it is sufficient to prove that $\partial_\alpha F(\alpha,u)$ is
positive when $u \in(0,1)$ and negative for $u > 1$ (proviso that $\alpha > 0$). Let $\psi$ be the
logarithmic derivative of the gamma function, then
\[
\partial_\alpha F(\alpha,u)
= \psi(1+\alpha) - \psi(1+u+\alpha) + \frac{u}{u+\alpha}
= \frac{u}{u+\alpha}
  -\sum_{n=1}^{\infty}\frac{u}{(n+\alpha)(n+u+\alpha)},
\]
where the series comes from the representation of gamma as Weierstrass product
(see~\cite[Th.~1.2.5]{specialFunctions}). Since
\[
\frac{u}{u+\alpha}
= u\sum_{n=1}^\infty \int_{u+n-1}^{u+n}\frac{\dd x}{(x+\alpha)^2}
=  \sum_{n=1}^\infty \frac{u}{(n + u + \alpha - 1)(n + u + \alpha)},
\]
to conclude it is sufficient to see that for every $n$
\begin{align*}
\frac{u}{(n + u + \alpha - 1)(n + u + \alpha)}
\geq \frac{u}{(n + \alpha)(n + u + \alpha)}
\end{align*}
if and only if $u < 1$.
\end{proof}
Howard and Shep~\cite{HowardSchep} proposed general theorems to estimate the norms of positive operators
in $L^p$ spaces and used them to compute the exact value of the norm of Volterra operator $V = T_1$ when
it is considered as a map $L^p[0,1]\to L^q[0,1]$ for every pair of indexes $p$, $q$. Our attempts to
estimate the norms of $T_\alpha$ via these results produced values which are larger than what we have
stated in Theorem~\ref{th:2A} and that we have proved using only basic tools. This is due to the fact
that the results in~\cite{HowardSchep} depend on a convenient choice of a test function, for which we
have not been able to find a good analogue for the general $T_\alpha$ operator. Also the strategy
allowing to compute $\|V\|_{p,q}$ fails for $T_\alpha$ operators, since the equation which should be
solved explicitly to detect the best test function becomes a very complicated integro-differential
equation in case $\alpha \neq 1$ or when $p$ and $q$ are not $2$: for the case $p=q=2$, however, we can
compute the norm of $T_\alpha$ as a byproduct of the study of $T_\alpha^* T_\alpha^\ph$ in
Section~\ref{sec:5A}, see Corollary~\ref{cor:1A}.

\section{Spectrum}\label{sec:3A}
When $q=p$ one can consider the spectrum of $T_\alpha\colon L^p[0,1] \to L^p[0,1]$. Barnes~\cite{Barnes}
investigated a general family of Volterra operators and showed that their spectrum is $\{0\}$, so that
they have at most $0$ as eigenvalue. The operators $T_{\alpha}$ with $\alpha\geq 1$ belong to this family
but the results in~\cite{Barnes} do not cover the case $\alpha < 1$. In fact, we prove the following
result.
\begin{theorem}\label{th:3A}
Fix $p\in(1,+\infty)$. Suppose that $\alpha \geq 1$, then
\[
\sigma(T_\alpha)=\{0\},
\qquad
\sigma_0(T_\alpha)=\emptyset.
\]
Suppose that $\alpha < 1$, then
\[
\sigma(T_\alpha) = \sigma_0(T_\alpha)\cup\{0\},
\qquad
\sigma_0(T_\alpha) = \{\alpha^n(1-\alpha)\colon n\in\N\}
\]
and the eigenspace associated with $\alpha^n(1-\alpha)$ is generated by $x^{\frac{\alpha}{1-\alpha}}
P_{n,\alpha}(\log x)$ where $P_{n,\alpha}$ is a suitable polynomial with degree $n$ and depending on
$\alpha$. %
The span of the family $\{f_n\}_{n\in\N}$ is dense in $L^p[0,1]$.
\end{theorem}
\noindent %
The proof of this theorem also provides a formula for $P_{n,\alpha}$, see~\eqref{eq:20A}.\\
Note that $T_\alpha$ is not normal, so that the density of the span of the eigenfunctions is not
sufficient to prove that they form a Schauder basis for the space; the known conditions that are capable
of guaranteeing this property do not appear to be practically verifiable in the present case
(see~\cite[Proposition~1.a.3]{LindenstraussTzafriri}, and~\cite[Chapter~VI]{GohbergKrein}).
\medskip

For the proof of this theorem we need a preliminary study of the properties of the iterations of
$T_\alpha$. This is possible since for every $f\in L^p[0,1]$, $T_\alpha f$ can be written as
$\int_{[0,1]} \chi_{[0,x^\alpha]}(y) f(y) \dd y$, so that the iterations of $T_\alpha$ are
\[
(T_\alpha^n f)(x)
=
\int_{[0,1]} K_n(x,y) f(y) \dd y
\qquad
\forall n\geq 1,
\]
with kernels $K_n(x,y)$ satisfying the recursive formula:
\[
 K_1(x,y)= \chi_{[0,x^\alpha]}(y),
 \qquad
 K_{n+1}(x,y) = \int_{[0,1]} K(x,s)K_n(s,y) \dd s
\quad
\forall n\geq 1.
\]
The following proposition gives a convenient formula for $K_n$.
\begin{proposition}\label{prop:2A}
Assume $\alpha>0$. For every $n\in\N$ let $a_n$ and $b_n$ be defined as
\[
a_n := \frac{\alpha-\alpha^n}{1-\alpha}
\quad \text{and}\quad
b_1 := 1,
\quad
b_n:= \prod_{k=1}^{n-1}\frac{1-\alpha}{1-\alpha^k}
\quad
\text{when $n\geq 2$};
\]
when $\alpha=1$ these formulas have to be considered as limits, giving $a_n=n-1$ and $b_{n} = 1/(n-1)!$
in that case. Then
\[
K_n(x,y) = b_n \chi_{[0,x^{\alpha^n}]}(y) x^{a_n} g_n(x^{-\alpha^n}y),
\]
where
\[
g_1(z) := 1,
\quad\text{and}\quad
g_{n+1}(z) := (a_n+1)\int_{z^{1/\alpha^n}}^1 w^{a_n} g_n(w^{-\alpha^n}z)\dd w
\quad \forall n\geq 1.
\]
\end{proposition}
For example,
\begin{gather*}
g_2(z)  = 1 - z^{1/\alpha},
\qquad
g_3(z)  = 1 - (\alpha{+}1) z^{1/\alpha}
            + \alpha       z^{1/\alpha + 1/\alpha^2},\\
g_4(z)  = 1 - (\alpha^2{+}\alpha{+}1)        z^{1/\alpha}
            + (\alpha^3{+}\alpha^2{+}\alpha) z^{1/\alpha + 1/\alpha^2}
            - \alpha^3                       z^{1/\alpha + 1/\alpha^2 + 1/\alpha^3}.
%g_5(z) &= 1 - (\alpha^3{+}\alpha^2{+}\alpha{+}1)                    z^{\frac{1}{\alpha}}
%            + (\alpha^5{+}\alpha^4{+}2\alpha^3{+}\alpha^2{+}\alpha) z^{\frac{1}{\alpha} + \frac{1}{\alpha^2}}
%            - (\alpha^6{+}\alpha^5{+}\alpha^4{+}\alpha^3)           z^{\frac{1}{\alpha} + \frac{1}{\alpha^2} + \frac{1}{\alpha^3}}
%            + \alpha^6                                              z^{\frac{1}{\alpha} + \frac{1}{\alpha^2} + \frac{1}{\alpha^3} + \frac{1}{\alpha^4}}.
\end{gather*}
Functions $\{g_n\}_n$ satisfy also the relation
\begin{equation}\label{eq:6A}
g_{n+1}(z) = g_n(z) - \alpha^{n-1}z^{1/\alpha}g_n(z^{1/\alpha})
\end{equation}
for every $n$, which comes from the relation
\begin{equation}\label{eq:7A}
K_{n+1}(x,y)
= \frac{1}{a_n+1}[x^\alpha K_n(x^\alpha,y) - \alpha^{n-1}y^{1/\alpha}K_n(x,y^{1/\alpha})].
\end{equation}
The equality in~\eqref{eq:6A} can be used to prove the following explicit formulas for $g_n$ and $K_n$,
again valid for every $n$:
\begin{align}
g_n(z)   &= \sum_{k=0}^{n-1} (-1)^k \Gbinom{n{-}1}{k}_\alpha \alpha^{\binom{k}{2}} z^{\frac{1-\alpha^{-k}}{\alpha-1}},
                                                                                                     \label{eq:8A}\\
K_n(x,y) &= b_n \chi_{[0,x^{\alpha^n}]}(y)
            \sum_{k=0}^{n-1} (-1)^k \Gbinom{n{-}1}{k}_\alpha \alpha^{\binom{k}{2}}
                             x^{\frac{\alpha^{n-k}-\alpha}{\alpha-1}} y^{\frac{1-\alpha^{-k}}{\alpha-1}},
                                                                                                     \label{eq:9A}
\end{align}
where $\Gbinom{m}{k}_\alpha := \frac{(1-\alpha^{m})(1-\alpha^{m-1})\cdots(1-\alpha^{m-k+1})}
{(1-\alpha)(1-\alpha^2)\cdots(1-\alpha^k)}$. This is the so called $\alpha$-analogue of the binomial
coefficient and in spite of its definition it is a polynomial with integral and positive coefficients so
that in particular it is positive when $\alpha >0$ (see~\cite[Ch.~10]{specialFunctions}).
Formulas~\eqref{eq:8A}-\eqref{eq:9A} generalize to every $\alpha$ the binomial presentations of
identities $g_n(z) = (1-z)^{n-1}$ and $K_n(x,y) = \frac{\chi_{[0,x]}(y)}{(n-1)!} (x-y)^{n-1}$ for
$\alpha=1$. They are useful in case one wants to graph $g_n$ and $K_n$ for some $n$, but the alternating
signs appearing there make them not useful to produce lower/upper bounds, which however is our main
interest. For this reason we do not prove relations~\eqref{eq:6A}--\eqref{eq:9A} here.
\begin{proof}[Proof of Proposition~\ref{prop:2A}]
The claim is evident for $n=1$. By inductive hypothesis
\begin{align*}
K_{n+1}(x,y)
&= \int_0^1 K_1(x,s)K_n(s,y) \dd s
 = \int_0^1 \chi_{[0,x^\alpha]}(s) \chi_{[0,s^{\alpha^n}]}(y) b_n s^{a_n} g_n(s^{-\alpha^n}y) \dd s.
\intertext{The product $\chi_{[0,x^\alpha]}(s) \chi_{[0,s^{\alpha^n}]}(y)$ is $1$ if and only if $y\leq
x^{\alpha^{n+1}}$ and $s \in [y^{1/\alpha^n},x^\alpha]$, thus}
K_{n+1}(x,y)
&= b_n \chi_{[0,x^{\alpha^{n+1}}]}(y)\int_{y^{1/\alpha^n}}^{x^\alpha} s^{a_n} g_n(s^{-\alpha^n}y) \dd s.
\intertext{Setting $s=x^\alpha w$ this is }
K_{n+1}(x,y)
&= b_n \chi_{[0,x^{\alpha^{n+1}}]}(y)x^{\alpha (a_n+1)}\int_{x^{-\alpha}y^{1/\alpha^n}}^{1} w^{a_n} g_n(w^{-\alpha^n} x^{-\alpha^{n+1}} y)\dd w,
\end{align*}
which is the claim since $a_{n+1}=\alpha(a_n+1)$ and $b_n = b_{n+1}(a_n+1)$.
\end{proof}
Now we can prove Theorem~\ref{th:3A}.
\begin{proof}[Proof of Theorem~\ref{th:3A}]
From the recursive formula for $g_n$ it is evident that $g_n(z)\in[0,1]$ for every $z\in[0,1]$ and every
$n$, so that
\[
0
\leq K_n(x,y)
\leq b_n \cdot \chi_{[0,x^{\alpha^n}]}(y)\cdot x^{a_n}.
\]
Thus,
\begin{equation}\label{eq:10A}
\|T_\alpha^n\|_{p,p}\leq \|K_n\|_{L^p([0,1]\times[0,1])}
\leq b_n (p a_n + \alpha^n + 1)^{-1/p}
\leq b_n
\end{equation}
and Gelfand's formula for the spectral radius $\rho(T_\alpha)$ produces the bound
\begin{equation}\label{eq:11A}
\rho(T_\alpha)
= \lim_{n\to\infty} \|T_\alpha^n\|_{p,p}^{1/n}
\leq \lim_{n\to\infty} b_n^{1/n}
= \begin{cases}
    1- \alpha   & \text{if } \alpha \leq 1,\\
    0           & \text{if } \alpha \geq 1.
  \end{cases}
\end{equation}
The inclusion $L^p[0,1]\subset L^1[0,1]$ assures that $T_\alpha f$ is absolutely continuous in $[0,1]$
and $\CC^\infty$ in $(0,1)$, so that $T_\alpha$ cannot be surjective. However it is injective, by the
Lebesgue differentiability Theorem. Thus, the spectrum $\sigma$ and the point spectrum $\sigma_0$ in case
$\alpha\geq 1$ are
\[
\sigma(T_\alpha)=\{0\},
\qquad
\sigma_0(T_\alpha)=\emptyset,
\]
respectively, and the problem is completely settled in this case.
\medskip\\
Assume now that $\alpha\in(0,1)$. One verifies that $x^{\frac{\alpha}{1-\alpha}}\in L^p[0,1]$ is an
eigenfunction for $T_\alpha$, with eigenvalue $1-\alpha$. With~\eqref{eq:11A} this proves that the
spectral radius $\rho(T_\alpha)$ equals $1-\alpha$, and that $\sigma_0(T_\alpha)$ contains $1-\alpha$.
The special form of this eigenfunction for the eigenvalue $1-\alpha$ suggests to write the generic
eigenfunction $f(x)$ as $x^{\frac{\alpha}{1-\alpha}} h(\log x)$. The regularity of $f$ shows that $h$ is
in $\CC^\infty(-\infty,0)$ and admits a continuous extension to $0$ from the left. In terms of $h$ the
equation $T_\alpha f = \lambda f$ becomes
\begin{equation}\label{eq:12A}
\int_{-\infty}^{\alpha y}e^{\frac{w}{1-\alpha}} h(w)\dd w = \lambda e^{\frac{\alpha y}{1-\alpha}} h(y).
\end{equation}
A derivation with respect to $y$ of this identity  and a rearrangement of the terms produce the equation
\begin{equation}\label{eq:13A}
h'(y) = A h(\alpha y) - B h(y)
\qquad
\text{with $A:=\frac{\alpha}{\lambda}$ and $B:= \frac{\alpha}{1-\alpha}$},
\end{equation}
and the evaluation of the equation at $y=0$ produces the condition
\begin{equation}\label{eq:14A}
\int_{-\infty}^{0}e^{\frac{w}{1-\alpha}} h(w)\dd w = \lambda h(0).
\end{equation}
\noindent %
On the contrary, every function $h$ in $\CC^\infty(-\infty,0)$ admitting a finite limit to $0^\sm$ and
satisfying both~\eqref{eq:13A} and~\eqref{eq:14A} also satisfies~\eqref{eq:12A} and hence produces an
eigenfunction.\\
By induction on the order $k$, one proves that there exist constants $q_{k,j}$ with $j\in\Z$ (depending
on parameters $A$ and $B$) with $q_{k,j}=0$ when $j<0$ or $j>k$, such that
\begin{equation}\label{eq:15A}
(-1)^k h^{(k)}(y) = \sum_{j=0}^k (-1)^j q_{k,j} h(\alpha^j y).
\end{equation}
In fact, the formula holds for $k=1$ with $q_{1,0} := B$ and $q_{1,1} := A$. Assume that the formula
holds for $k$. Deriving the formula and plugging~\eqref{eq:13A} into the resulting identity we see that
\begin{align*}
(-1)^{k+1} h^{(k+1)}(y)
&=-\sum_{j=0}^k (-1)^j q_{k,j} \alpha^j h'(\alpha^j y)  \\
&=-\sum_{j=0}^k (-1)^j q_{k,j} \alpha^j \big(A h(\alpha^{j+1} y) - B h(\alpha^j y)\big)\\
&=-\sum_{j=0}^k (-1)^j A q_{k,j} \alpha^j h(\alpha^{j+1} y)
  +\sum_{j=0}^k (-1)^j B q_{k,j} \alpha^j h(\alpha^j y)\\
&= \sum_{j=1}^{k+1} (-1)^j A q_{k,j-1} \alpha^{j-1} h(\alpha^j y)
  +\sum_{j=0}^k (-1)^j B q_{k,j} \alpha^j h(\alpha^j y)
\end{align*}
so that the formula for $k+1$ emerges once we define
\[
q_{k+1,j} := A \alpha^{j-1} q_{k,j-1} + B \alpha^j q_{k,j}
\qquad  \forall j.
\]
Let $C:=\max(|A|,|B|)$. The recursive definition of constants $q_{k,j}$ shows that
\[
|q_{k,j}| \leq C^k \alpha^{\binom{j}{2}} \Gbinom{k}{j}_\alpha
\qquad  \forall k,j,
\]
where $\Gbinom{k}{j}_\alpha$ is the already mentioned Gaussian binomial coefficient (here we use the
relation $\Gbinom{k}{j}_\alpha = \Gbinom{k-1}{j-1}_\alpha + \alpha^j\Gbinom{k-1}{j}_\alpha$,
see~\cite[Eq.~10.0.3]{specialFunctions}).
By~\eqref{eq:15A} we see that every derivative $h^{(k)}$ admits a continuation to $0$ from the left, with
value
\begin{align}
h^{(k)}(0^\sm) &= (-1)^k  h(0^\sm) \sum_{j=0}^k (-1)^j q_{k,j},                                    \label{eq:16A}
\intertext{and that}
|h^{(k)}(y)|   &\leq C^k \sum_{j=0}^k \alpha^{\binom{j}{2}} \Gbinom{k}{j}_\alpha |h(\alpha^j y)|.  \label{eq:17A}
\end{align}
Let $R$ be any positive parameter. We are assuming that $\alpha <1$, therefore $\alpha^j y$ is in
$[-R,0]$ whenever $y$ is in $[-R,0]$ and $j\geq 0$. As a consequence, by~\eqref{eq:17A} and denoting
$\|\cdot\|_{\infty,R}$ the sup norm in $[-R,0]$, we see that
\begin{align*}
\|h^{(k)}\|_{\infty,R}
\leq C^k \sum_{j=0}^k \alpha^{\binom{j}{2}} \Gbinom{k}{j}_\alpha \|h\|_{\infty,R}
=    C^k \|h\|_{\infty,R} \prod_{j=0}^{k-1}(1+\alpha^j),
\end{align*}
where the equality comes from the identity $\sum_{j=0}^k \alpha^{\binom{j}{2}} \Gbinom{k}{j}_\alpha t^j =
\prod_{j=0}^{k-1}(1+\alpha^j t)$ (see~\cite[Eq.~10.0.9]{specialFunctions}). The full product $c(\alpha)
:= \prod_{j=0}^{\infty}(1+\alpha^j)$ converges, because $\alpha < 1$, therefore we can conclude that
\begin{align*}
\|h^{(k)}\|_{\infty,R}
\leq  c(\alpha) \|h\|_{\infty,R} C^k.
\end{align*}
This shows that $\|h^{(k)}\|_{\infty,R}$ diverges with the order $k$ as $C^k$, at most.
Since $R$ is arbitrary, this suffices to assure that the power series $\sum_{k=0}^\infty
\frac{h^{(k)}(0)}{k!} y^k$ converges to $h$ for $y\leq 0$ and provides an analytic extension of $h$ as an
entire function. \\
%
% NOTICE:
% Let $T_m(y):= \sum_{k=0}^m \frac{h^{(k)}(0)}{k!} y^k$. Taylor formula with integral remainder
% \[
% h(y) = T_{m-1}(y) + \int_{0}^{y} h^{(m)}(w) \frac{(y-w)^{m-1}}{(m-1)!}\dd w
% \]
% holds for every $m$ and for every $y \leq 0$, because $h \in C^\infty((-\infty,0])$. %
% This formula and the bound for $h^{(m)}$ show that
% \[
% \|h - T_{m-1}\|_{\infty,R}
% \leq c(\alpha) \|h\|_{\infty,R} \frac{(CR)^{m}}{m!}
% \]
% and this proves the equality
% \[
% h(y) = \lim_{m\to\infty} T_{m-1}(y) = \sum_{k=0}^\infty \frac{h^{(k)}(0)}{k!} y^k,
% \]
% for every $y\leq 0$, since $R$ is arbitrary. This representation also provides a continuation of $h$ to
% $\C$ as an entire function.
%
We also need an explicit formula for the coefficients of the power series representing $h$. They can be
recovered from~\eqref{eq:16A} but the following argument is quicker. In fact, writing $h(y) =
\sum_{k=0}^\infty \beta_k y^k$, the relation in~\eqref{eq:13A} readily shows that
\[
\beta_{k+1} = \frac{A\alpha^k - B}{k+1} \beta_k
\qquad \forall k\geq 0.
\]
By homogeneity we can select $\beta_0 = 1$, because every non zero function is a multiple of what we get
under this assumption. This yields
\begin{equation}\label{eq:18A}
h^{(k)}(0)
= k! \beta_k
= \prod_{j=0}^{k-1}(A\alpha^j - B)
%= \prod_{j=0}^{k-1}(-B)\Big(1-\frac{A}{B}\alpha^j\Big)
= \Big(\frac{-\alpha}{1-\alpha}\Big)^k\prod_{j=0}^{k-1}\Big(1 - \frac{\alpha^j(1-\alpha)}{\lambda}\Big).
\end{equation}
For it to produce an eigenfunction $h$ must satisfy also the condition~\eqref{eq:14A}, therefore
$\lambda$ must be a solution of the equation
\begin{align*}
\lambda
&= \int_{-\infty}^0 e^{\frac{w}{1-\alpha}}h(w) \dd w
 = \int_{-\infty}^0 e^{\frac{w}{1-\alpha}}\sum_{k=0}^\infty \beta_k w^k \dd w
 = \sum_{k=0}^\infty (-1)^k k!\beta_k (1-\alpha)^{k+1},
\end{align*}
where the exchange of the integral and the series is made possible by Fubini's theorem and~\eqref{eq:18A}
which shows that $k!|\beta_k| \ll_{\lambda,\alpha} \big(\frac{\alpha}{1-\alpha}\big)^k$ so that the
resulting series converges absolutely. %
Plugging~\eqref{eq:18A} into the previous equation we get that
\[
\lambda
= (1-\alpha) + (1-\alpha)\sum_{k=1}^\infty \alpha^k \prod_{j=0}^{k-1}\Big(1 - \frac{\alpha^j(1-\alpha)}{\lambda}\Big).
\]
Setting $(1-\alpha)/\lambda =: U$ it becomes
\begin{equation}\label{eq:19A}
0 = 1- U - \alpha U\sum_{k=0}^\infty \alpha^k \prod_{j=0}^{k}(1 - \alpha^j U).
\end{equation}
The quantity appearing to the right hand side coincides with $\prod_{j=0}^\infty (1-\alpha^j U)$:
% NOTICE:
in terms of $\alpha$-analogues symbols, this claim corresponds to the equality $\sum_{k=0}^\infty
(z;\alpha)_k \alpha^k = (1 - (z;\alpha)_\infty)/z$. We strongly suspect that this formula is well known
to every specialist in this area, but we have not been able to locate it precisely in literature. Thus,
we provide here a quick proof. Let
\[
F(z) := 1- z - \alpha z\sum_{k=0}^\infty \alpha^k \prod_{j=0}^{k}(1 - \alpha^j z).
\]
Once again we notice that the inner product can be bounded uniformly in $k$ and $|z|\leq 1$, so that
$F(z) \to 1$ as $z\to 0$. In the series defining $F$ we separate the term with $k=0$ and in each other
term with $k\geq 1$ we collect the term $1-z$ coming from the case $j=0$. This produces the identity
\begin{align*}
F(z)
&= 1- z - \alpha z(1-z)  - \alpha z (1-z)\Big(\sum_{k=1}^\infty \alpha^k \prod_{j=1}^{k}(1 - \alpha^j z)\Big)\\
&= (1- z)\Big(1 - \alpha z - \alpha^2 z \sum_{k=0}^\infty \alpha^k \prod_{j=0}^{k}(1 - \alpha^j (\alpha z))\Big)
 = (1- z)F(\alpha z).
\end{align*}
Iterating this identity we get that $F(z)=\prod_{j=0}^{L-1}(1-\alpha^j z) F(\alpha^L z)$, for every $L$.
Setting $L\to\infty$ we get the equality $F(z) = \prod_{j=0}^\infty (1-\alpha^j z)$, since $F(\alpha^L z)
\to 1$ as $L\to\infty$.\\
In this way we see that Equation~\eqref{eq:19A} actually says that
\begin{align*}
0 = \prod_{j=0}^\infty (1-\alpha^j U).
\end{align*}
The product converges absolutely, hence its unique zeros come from zero factors and this forces $U =
\alpha^{-n}$ for $n\in\N$ i.e. $\lambda=\alpha^n(1-\alpha)$. This proves that the unique (possible)
eigenvalues are the numbers $\alpha^n(1-\alpha)$. \\
Moreover, suppose that $\lambda = \alpha^n(1-\alpha)$ for a given $n\in\N$. Formula~\eqref{eq:18A} shows
that in this case $h^{(k)}(0)=0$ whenever $k>n$, so that the previous computations show that the
eigenfunction has the shape $x^{\frac{\alpha}{1-\alpha}} h(\log x)$ with $h$ which is actually a
polynomial with degree $\leq n$. The same formula also allows an explicit presentation for this
polynomial:
\begin{align*}
h(y)
&= \sum_{k=0}^n h^{(k)}(0)\frac{y^k}{k!}
 = \sum_{k=0}^n \Big(\frac{-\alpha}{1-\alpha}\Big)^k \prod_{j=0}^{k-1}(1-\alpha^{j-n})\frac{y^k}{k!}\\
&= \sum_{k=0}^n \Big(\prod_{j=0}^{k-1}\frac{1-\alpha^{j-n}}{1-\alpha^{-1}}\Big)\frac{y^k}{k!}
\end{align*}
showing that its degree is $n$, and giving the eigenfunction
\begin{equation}\label{eq:20A}
f(x)
= x^{\frac{\alpha}{1-\alpha}}
  \sum_{k=0}^n \Big(\prod_{j=0}^{k-1}\frac{1-\alpha^{j-n}}{1-\alpha^{-1}}\Big)\frac{(\log x)^k}{k!}.
\end{equation}
Finally, the eigenfunctions we have found generate the vector space
\[
\mathcal{P}
:= \big\{x^{\frac{\alpha}{1-\alpha}}P(\log x),\ P \in \C[z]\big\}.
\]
Let $f\in L^p[0,1]$. The change of variable $x:= e^w$ gives the identity
\[
\int_0^1 \big|f(x) - x^{\frac{\alpha}{1-\alpha}}P(\log x)\big|^p \dd x
= \int_{-\infty}^0 e^{\frac{1+(p-1)\alpha}{1-\alpha} w}\big| e^{\frac{-\alpha}{1-\alpha}w}f(e^w) - P(w)\big|^p \dd w.
\]
The case $P=0$ shows that $e^{\frac{-\alpha}{1-\alpha}w}f(e^w)$ is in
$L^p((-\infty,0],e^{\frac{1+(p-1)\alpha}{1-\alpha} w} \dd w)$. Since polynomials are dense in this space,
see~\cite[p.~40]{Szego}, the previous computation also shows that $\mathcal{P}$ is dense in $L^p[0,1]$.
\end{proof}

\section{Norm of iterates}\label{sec:4A}
The study of the behaviour of the norm of the $n$-th iterated of $V = T_1$ when $n$ diverges is a
classical problem which in the space $L^2[0,1]$ has been solved by Lao and Whitley~\cite{LaoWhitley} for
the order and by Kershaw~\cite{Kershaw2} for the asymptotic: $\|n! V^n\|_{2,2} \to 1/2$. See
also~\cite{LittleReade} for an elementary proof and \cite{BottcherDorfler} for a proof with explicit
bounds. Later Eveson extended this result to operators of the form $\int_0^x k(x-s)f(s) \dd s$ under mild
conditions for the kernel both in the $L^2[0,1]$ case~\cite{Eveson} and for the general $L^p[0,1]$
case~\cite{Eveson2}. Adell and Gallardo--Gutierrez~\cite{AdellGallardo-Gutierrez} proved explicit bounds
for each $s$-th Riemann--Liouville fractional integration operators (which coincides with the $s$-th
iteration of $V$ when $s$ is an integer) in the $L^p[0,1]$ case. \\
We prove a similar result for operators $T_\alpha$ when considered as map from the space $L^p[0,1]$ in
itself.
\begin{theorem}\label{th:4A}
Fix $p\in(1,+\infty)$, $\alpha >0$, and let $n\to \infty$. Then
\[
\log \|T_\alpha^n\|_{p,p}
= \begin{cases}
    n\log(1-\alpha) + o_{p,\alpha}(n)              & \text{if } \alpha < 1 ,\\
    -n\log n + O_{p}(n)                            & \text{if } \alpha = 1 ,\\
    -\frac{1}{2}n^2\log \alpha + O_{p,\alpha}(n)   & \text{if } \alpha > 1 .
  \end{cases}
\]
\end{theorem}
\noindent %
Actually, by the work of the cited authors, in the case $\alpha=1$ the full expansion is known up to the
order $o_p(1)$, so that the conclusions in Theorem~\ref{th:4A} are less precise. Nevertheless, they
already show in a quantitative way the threshold effect associated with the passage of $\alpha$ through
$1$: the logarithm of the norm diverges linearly when $\alpha < 1$ and quadratically when $\alpha >1$.

The case $\alpha < 1$ comes from Gelfand's formula for the spectral radius and the case $\alpha=1$ is a
weak version of the computations in~\cite{LaoWhitley}, hence only the case $\alpha>1$ must be proved. For
this purpose we use the following lower bound for the functions $g_n$ appearing in Proposition~\ref{prop:2A}.
\begin{proposition}\label{prop:3A}
Let $\alpha > 0$ and $n\geq 2$. Then
\[
g_n(z) \geq (1 - z^{1/((n-1)\alpha)})^{n-1}
\qquad  \forall z\in [0,1].
\]
\end{proposition}
\begin{proof}
The inequality for $n=2$ holds as equality. Assume that the claim is true for $n$. Then, by the integral
relation in Proposition~\ref{prop:2A} one gets
\[
g_{n+1}(z)
\geq (a_n+1)\int_{z^{1/\alpha^n}}^1 w^{a_n}(1 - (zw^{-\alpha^n})^{1/{((n-1)\alpha)}})^{n-1} \dd w.
\]
Let $\gamma$ be a parameter in $[z^{1/\alpha^n},1]$ that we will choose later. Then
\begin{align*}
g_{n+1}(z)
&\geq (a_n+1)\int_{\gamma}^1 w^{a_n}(1 - (zw^{-\alpha^n})^{1/{((n-1)\alpha)}})^{n-1} \dd w \\
&\geq (1 - (z \gamma^{-\alpha^n})^{1/{((n-1)\alpha)}})^{n-1} (a_n+1)\int_{\gamma}^1 w^{a_n} \dd w \\
&=    (1 - (z \gamma^{-\alpha^n})^{1/{((n-1)\alpha)}})^{n-1} (1 - \gamma^{a_n+1}).
\end{align*}
We set $\gamma$ such that $\gamma^{a_n+1} = (z \gamma^{-\alpha^n})^{1/{((n-1)\alpha)}}$, i.e. $\gamma =
z^{1/((n-1)\alpha(a_n+1)+\alpha^n)}$. This value is in the allowed interval $[z^{1/\alpha^n},1]$, hence
\begin{align*}
g_{n+1}(z)
&\geq (1 - z^{(a_n+1)/((n-1)\alpha(a_n+1)+\alpha^n)})^n.
\end{align*}
The definition of $a_n$ implies that $(a_n+1)/((n-1)\alpha(a_n+1)+\alpha^n) \geq 1/(n\alpha)$ for
every $n$ (because it is equivalent to $a_n+1\geq \alpha^{n-1}$ which is true since $a_n+1 =
(\alpha^n-1)/(\alpha-1)$), so that
\begin{align*}
g_{n+1}(z)
&\geq (1 - z^{1/(n\alpha)})^n.
\end{align*}
This proves the claim by induction.
\end{proof}
\noindent %
Now we can prove the remaining case $\alpha >1$ in Theorem~\ref{th:4A}. By Propositions~\ref{prop:2A}
and~\ref{prop:3A} we get that
\[
K_n(x,y) \geq b_n
              \chi_{[0,x^{\alpha^n}]}(y) x^{a_n}
              (1-(x^{-\alpha^n}y)^{1/((n-1)\alpha)})^{n-1}.
\]
Let $f(x)=1$ for every $x$. Then $\|f\|_p = 1$, and
\begin{align*}
(T_\alpha^n f)(x)
&=    \int_0^1 K_n(x,y)\dd y
 \geq b_n x^{a_n}\int_0^{x^{\alpha^n}} \!\!\!\!\!\! (1- (yx^{-\alpha^n})^{1/((n-1)\alpha)})^{n-1}\dd y.
\intertext{Setting $z := (yx^{-\alpha^n})^{1/((n-1)\alpha)}$, i.e. $y=x^{\alpha^n} z^{(n-1)\alpha}$, this
becomes}
&=    b_n x^{a_n+\alpha^n} (n-1)\alpha\int_0^1 w^{(n-1)\alpha -1}(1-w)^{n-1}\dd z\\
&=    b_n x^{a_n+\alpha^n} (n-1)\alpha B((n-1)\alpha,n)
 =    b_n x^{a_n+\alpha^n} (n-1)\alpha\frac{\Gamma((n-1)\alpha)\Gamma(n)}{\Gamma((n-1)\alpha+n)},
\end{align*}
where we have used the representation of the Beta function in terms of gammas. As a consequence,
\begin{align*}
\big\|T_\alpha^n\big\|_{p,p}
&\geq \big\|T_\alpha^n f\big\|_p
 \geq \frac{(n-1)\alpha\, b_n}{(a_n p+\alpha^n p+1)^{1/p}}
      \frac{\Gamma((n-1)\alpha)\Gamma(n)}{\Gamma((n-1)\alpha+n)}.
\end{align*}
Using Stirling asymptotic formula~\cite[Th.~1.4.1]{specialFunctions} we deduce that there exists a
constant $c>1$, depending on $\alpha$ and $p$ but independent of $n$, such that
\begin{align*}
\big\|T_\alpha^n\big\|_{p,p}
\geq b_n c^{-n} n^{O_{\alpha,p}(1)}
\qquad \text{as $n\to\infty$}.
\end{align*}
By~\eqref{eq:10A} we already know that $\|T_\alpha^n\|_{p,p} \leq b_n$, hence $\|T_\alpha^n\|_{p,p} {=}
b_n \exp(O_{\alpha,p}(n))$. Everything proved up to now holds for any positive $\alpha$. Suppose that
$\alpha > 1$, then the definition of $b_n$ shows that in this case $b_n = \alpha^{-n^2/2}
\exp(O_{\alpha}(n))$, so that
\[
\|T_\alpha^n\|_{p,p} = \alpha^{-n^2/2}\exp(O_{\alpha,p}(n))
\qquad \text{as $n\to\infty$},
\]
as claimed.

\section{The case \texorpdfstring{$p=q=2$}{p=q=2}}\label{sec:5A}
When $T_\alpha\colon L^2[0,1] \to L^2[0,1]$ we can compute exactly spectrum and eigenfunctions for
$T^*_{\alpha}{T}^\ph_{\alpha}$ (see Theorem~\ref{th:5A}), in particular we get an exact formula for
$\|T_\alpha\|_{2,2}$ and its asymptotic for $\alpha\to 0$ and $\alpha \to\infty$ (see
Corollary~\ref{cor:1A}).
\begin{theorem}\label{th:5A}
Let
\begin{equation}\label{eq:21A}
H_\alpha(z) := \sum_{k=0}^\infty \frac{(-z)^k}{k!\prod_{j=1}^k(j-\frac{1}{1+\alpha})}.
\end{equation}
All zeros for $H_\alpha$ are real and positive. Let $\{h_{n}(\alpha)\}_n$ be the sequence of these zeros,
ordered by their size. The spectrum for $T^*_\alpha T^\ph_\alpha$ coincides with the set
\begin{equation}\label{eq:22A}
\sigma_0(T^*_\alpha T^\ph_\alpha)
= \Big\{\frac{\alpha}{(1+\alpha)^2} \frac{1}{h_n(\alpha)}\colon n\in\N\Big\},
\end{equation}
each eigenspace is one-dimensional and the $n$-th eigenspace is generated by the function
\begin{equation}\label{eq:23A}
\sum_{k=0}^\infty \frac{(- h_n(\alpha))^k}{k!\prod_{j=1}^k(j-\frac{1}{1+\alpha})} x^{\frac{(1+\alpha)k}{\alpha}}.
\end{equation}
\end{theorem}
\noindent %
This result extends Halmos' computation for the classic Volterra operator $V=T_1$, see~\cite[Problem~188,
p.~100]{Halmos}.
% P. Halmos, A Hilbert space problem book (Springer, New York, 1982)
% problem 188  p. 100
% hint p. 158
% solution p. 300
%
In fact, for $\alpha=1$ the function $H_1$ as given in~\eqref{eq:21A} becomes
\[
H_1(z)
= \sum_{k=0}^\infty \frac{(-z)^k}{k!\prod_{j=1}^k(j-\frac{1}{2})}
%= \sum_{k=0}^\infty \frac{(-4z)^k}{2^kk!\prod_{j=1}^k(2j-1)}
%= \sum_{k=0}^\infty \frac{(-4z)^k}{(2k)!!(2k-1)!!}
= \sum_{k=0}^\infty \frac{(-4z)^k}{(2k)!}
= \cos(2\sqrt{z})
\]
so that the spectrum~\eqref{eq:22A} becomes $\{\frac{4}{\pi^2}(1+2n)^{-2}, n \in \N\}$ and the
$n$-th eigenvalue gives the eigenfunction
\[
\sum_{k=0}^\infty \frac{(- \frac{\pi^2}{16}(1+2n)^2)^k}{k!\prod_{j=1}^k(j-\frac{1}{2})} x^{2k}
= \cos(\tfrac{\pi}{2}(1+2n) x).
\]
\begin{proof}
The operator $T^*_\alpha T^\ph_\alpha$ is compact, selfadjoint and strictly positive since
$T_\alpha$ is injective, thus the eigenvalue equation with $\lambda := 1/\omega^2$ says that
\begin{equation}\label{eq:24A}
 \omega^2\int_{x^{1/\alpha}}^1 \int_0^{z^\alpha} f(u) \dd u \dd z = f.
\end{equation}
We know that $T_\alpha f$ is absolutely continuous, thus the equation shows that $f\in C^0([0,1])\cap
C^1((0,1])$. The derivative of this equation produces the equality
\begin{equation}\label{eq:25A}
 -\frac{\omega^2}{\alpha} x^{\frac{1-\alpha}{\alpha}}\int_0^{x} f(u) \dd u = f'.
\end{equation}
Since $\frac{1}{x}\int_0^{x} f(u) \dd u$ goes to $f(0)$ as $x\to 0$, this equation shows both that
actually $f\in C^1([0,1])$, with $f'(0)=0$, and that $f\in C^2((0,1])$. A further derivation shows that
\begin{equation}\label{eq:26A}
 -\frac{(1-\alpha)\omega^2}{\alpha^2} x^{\frac{1-2\alpha}{\alpha}}\int_0^{x} f(u) \dd u
 -\frac{\omega^2}{\alpha} x^{\frac{1-\alpha}{\alpha}}f = f''.
\end{equation}
When $\alpha<1$, this equation shows that $f\in C^2([0,1])$, with $f''(0)=0$ (but for $\alpha=1$ only the
existence of $f''(0)$ can be deduced, and for $\alpha>1$ also the existence of $f''(0)$ is not evident).
Moreover, combining~\eqref{eq:25A} and~\eqref{eq:26A} we get the differential equation
\begin{equation}\label{eq:27A}
  f''
  - \frac{1-\alpha}{\alpha}\frac{f'}{x}
  + \frac{\omega^2}{\alpha} x^{\frac{1-\alpha}{\alpha}}f = 0.
\end{equation}
This is a homogeneous second order differential equation, with regular (i.e., analytic) coefficients in
$(0,1)$, and a singularity at $0$ coming from the quotient $f'/x$ and from the possible non-analyticity
of $x^{\frac{1-\alpha}{\alpha}}$. The first term is actually under control, since we know that the
solutions we are looking for have $f'(0)=0$; the second term is more difficult to deal with, but we can
improve it with a suitable change of variable. Suppose $f(x)= g(x^\beta)$, for some $\beta>0$. This
produces the equalities
\[
f' = \beta x^{\beta-1}g',
\qquad
f'' = \beta(\beta-1) x^{\beta-2}g' + \beta^2 x^{2\beta-2}g''
\]
which in~\eqref{eq:27A} give the equation
\[
  \beta^2 x^{2\beta-2}g''
  + \beta\Big(\beta - \frac{1}{\alpha}\Big) x^{\beta-2}g'
  + \frac{\omega^2}{\alpha} x^{\frac{1-\alpha}{\alpha}}g = 0,
\]
i.e.,
\[
  g''
  + \Big(1 - \frac{1}{\alpha\beta}\Big) x^{-\beta}g'
  + \frac{\omega^2}{\alpha\beta^2} x^{\frac{1+\alpha}{\alpha}-2\beta}g = 0.
\]
Thus, setting $\frac{1+\alpha}{\alpha}-2\beta = \beta$, i.e. $\beta := \frac{1+\alpha}{3\alpha}$, we get
the equation
\begin{equation}\label{eq:28A}
  g''
  - \frac{2-\alpha}{1+\alpha} \frac{g'}{z}
  + \frac{9\alpha\omega^2 z}{(1+\alpha)^2} g = 0,
\end{equation}
where derivatives are with respect to $z$ and $f(x) = g(x^{\frac{1+\alpha}{3\alpha}})$ (and $g(z) =
f(z^{\frac{3\alpha}{1+\alpha}})$). Note that $\beta=1$ in case $\alpha=1/2$ (so that in this case the
transformation of $x$ into $z$ is the identity), and that~\eqref{eq:28A} coincides with the equation
satisfied by Airy's function when $\alpha=2$ (after a suitable rescaling of the variable).\\
When written for the $g$ function, the eigenvalue equation~\eqref{eq:24A} reads
%\[
%\omega^2\int_{z^{1/(\alpha\beta)}}^1 \int_0^{s^{\alpha\beta}} v^{1/\beta-1} g(v) \dd v \dd s  = \beta g(z),
%\]
%i.e.,
\[
\omega^2\int_{z^{\frac{3}{1+\alpha}}}^1 \int_0^{s^{\frac{1+\alpha}{3}}} v^{\frac{2\alpha-1}{1+\alpha}} g(v) \dd v \dd s
= \frac{3\alpha}{1+\alpha} g(z).
\]
A derivation with respect to $z$ and a division by $z$ show that
\[
\omega^2 z^{\frac{1-2\alpha}{1+\alpha}}\int_0^z v^{\frac{2\alpha-1}{1+\alpha}} g(v) \dd v
= -\alpha \frac{g'}{z}.
\]
Restoring $f$ in this integral we get
\[
\omega^2\frac{1+\alpha}{3\alpha} z \cdot z^{\frac{-3\alpha}{1+\alpha}}\int_0^{z^{\frac{3\alpha}{1+\alpha}}} f(w) \dd w
= -\alpha \frac{g'}{z}.
\]
Since $f$ is continuous at $0$, this formula shows that $g'/z^2$ admits a finite limit as $z\to 0$.
In particular, both $g'(0)$ and $g''(0)$ exist and equal zero.\\
We look for a solution admitting a representation as power series $g(z)=\sum_{k=0}^\infty c_k z^k$.
Plugging the power series into~\eqref{eq:28A} and with the assumptions that $c_1=c_2=0$, and that $c_0=1$
(by homogeneity), we get
\[
c_0 = 1,
\quad
c_1 =c_2 = 0,
\qquad
(k+3)\Big(k + \frac{3\alpha}{1+\alpha}\Big)c_{k+3}
  = - \frac{9\alpha\omega^2}{(1+\alpha)^2} c_k
\quad
\forall\ k \in\N.
\]
Iterating the recursion, we get for the coefficients the explicit formula:
\[
c_{3k} = \frac{(- \frac{\alpha\omega^2}{(1+\alpha)^2})^k}{k!\prod_{j=1}^k(j-\frac{1}{1+\alpha})}
\quad\forall\ k\in\N,
\qquad
c_{\ell} = 0 \quad \text{otherwise}.
\]
These coefficients produce the function
\[
G_0(z)
:= \sum_{k=0}^\infty \frac{(- \frac{\alpha\omega^2}{(1+\alpha)^2})^k}{k!\prod_{j=1}^k(j-\frac{1}{1+\alpha})} z^{3k}
\]
which converges everywhere and therefore is a true solution of~\eqref{eq:28A}.
In terms of $x$, this produces the function
\[
f_0(x)
:= G_0(x^{\frac{1+\alpha}{3\alpha}})
 = \sum_{k=0}^\infty \frac{(- \frac{\alpha\omega^2}{(1+\alpha)^2})^k}{k!\prod_{j=1}^k(j-\frac{1}{1+\alpha})} x^{\frac{(1+\alpha)k}{\alpha}}.
\]
Now we produce a second and independent solution for~\eqref{eq:28A}. Since $G_0(0)=1$, locally it is not
$0$. Thus, an independent solution of~\eqref{eq:28A} can be obtained setting $g=G_0 R$ for a suitable $R$
and solving the resulting equation for $R$. After some computations, the new solution $G_1$ appears as
\[
R(z) := \int_0^z s^{\frac{2-\alpha}{1+\alpha}} G_0(s)^{-2} \dd s
\quad,\quad
G_1(z) := G_0(z) R(z)
        = G_0(z) \int_0^z s^{\frac{2-\alpha}{1+\alpha}} G_0(s)^{-2} \dd s.
\]
%
% NOTICE:
% This representation is again a power series in $z$ only when $\frac{2-\alpha}{1+\alpha} \in \N$, and
% this happens only when $\alpha=2$ or $\alpha=1/2$.
%
The general solution of~\eqref{eq:28A} is a linear combination $a G_0+b G_1$ with $a,b\in\R$, but only
the solutions with $b=0$ have a chance to produce eigenfunctions of $T^*_\alpha T^\ph_\alpha$. In fact,
we have proved that $g'/z^2$ admits a finite limit as $z\to 0$ whenever $g$ is an eigenfunction. The
function $G_0$ satisfies this property, hence the combination $aG_0 + bG_1$ with any $b\neq 0$ has this
property if and only if $G_1$ does the same. We have
\[
\frac{G_1'(z)}{z^2}
= \frac{G_0'(z)}{z^2}R(z) + G_0(z)\frac{R'(z)}{z^2}
= \frac{G_0'(z)}{z^2}R(z) + z^{\frac{-3\alpha}{1+\alpha}}G_0^{-1}(z).
\]
Here the first term has a finite limit, but the second diverges for every $\alpha>0$, so no combination
with $b\neq 0$ can be an eigenfunction.\\
Moreover, Equation~\eqref{eq:24A} shows that every eigenfunction has a zero at $x=1$, so that for $f_0(x)
= G_0(x^\frac{1+\alpha}{3\alpha})$ to be an eigenfunction it is necessary to have
\begin{equation}\label{eq:29A}
\sum_{k=0}^\infty \frac{(- \frac{\alpha\omega^2}{(1+\alpha)^2})^k}{k!\prod_{j=1}^k(j-\frac{1}{1+\alpha})}
= 0
\end{equation}
so that $\frac{\alpha\omega^2}{(1+\alpha)^2}$ is a zero for the function $H_\alpha$ as given
in~\eqref{eq:21A}.
We prove now that this condition is also sufficient. In fact,
\begin{align*}
\omega^2\int_{x^{1/\alpha}}^1 \int_0^{z^\alpha} f_0(u) \dd u \dd z
&= \omega^2\int_{x^{1/\alpha}}^1 \int_0^{z^\alpha}
    \Big[\sum_{k=0}^\infty \frac{(- \frac{\alpha\omega^2}{(1+\alpha)^2})^k}{k!\prod_{j=1}^k(j-\frac{1}{1+\alpha})} u^{\frac{(1+\alpha)k}{\alpha}}
    \Big] \dd u \dd z.
\intertext{Everything converges absolutely, thus exchanging the integrals and the series we get}
&= \sum_{k=0}^\infty \frac{(- \frac{\alpha\omega^2}{(1+\alpha)^2})^k \omega^2}{k!\prod_{j=1}^k(j-\frac{1}{1+\alpha})}
   \int_{x^{1/\alpha}}^1 \int_0^{z^\alpha} u^{\frac{(1+\alpha)k}{\alpha}} \dd u \dd z.
%&= \sum_{k=0}^\infty \frac{(- \frac{\alpha\omega^2}{(1+\alpha)^2})^k \omega^2}{k!\prod_{j=1}^k(j-\frac{1}{1+\alpha})}
%   \int_{x^{1/\alpha}}^1 \frac{\alpha}{(1+\alpha)k+\alpha}z^{(1+\alpha)k+\alpha} \dd z.
%
\intertext{The double integration gives}
&= \sum_{k=0}^\infty \frac{(- \frac{\alpha\omega^2}{(1+\alpha)^2})^k}{k!\prod_{j=1}^k(j-\frac{1}{1+\alpha})}
  \frac{\frac{\alpha \omega^2}{(1+\alpha)^2}}{(k+1)(k+\frac{\alpha}{1+\alpha})} [1-x^{\frac{1+\alpha}{\alpha}(k+1)}]\\
&=-\sum_{k=1}^\infty \frac{(- \frac{\alpha\omega^2}{(1+\alpha)^2})^k}{k!\prod_{j=1}^k(j-\frac{1}{1+\alpha})}
  +\sum_{k=1}^\infty \frac{(- \frac{\alpha\omega^2}{(1+\alpha)^2})^k}{k!\prod_{j=1}^k(j-\frac{1}{1+\alpha})} x^{\frac{1+\alpha}{\alpha}k},
\end{align*}
which is $f_0(x)$ whenever $\omega$ satisfies~\eqref{eq:29A}, so that $f_0$ is an eigenfunction.
Formula~\eqref{eq:23A} produces $f_0$ in terms of zeros for $H_\alpha$, once~\eqref{eq:29A} is taken
account.\\
Finally, we know that eigenvalues exist and must be real and positive, since $T^*_\alpha T^\ph_\alpha$ is
compact, self-adjoint and injective. However, all computations we have done need only the fact that
eigenvalues are not zero. In particular, every zero than $H_\alpha$ has in $\C$ produces an eigenvalue
for $T^*_\alpha T^\ph_\alpha$. This proves that all complex zeros for $H_\alpha$ are actually real and
positive.
\end{proof}
\begin{corollary}\label{cor:1A}
Let $h_0(\alpha)$ be the smallest positive zero for $H_\alpha$, as given in~\eqref{eq:21A}. Then
\begin{equation}\label{eq:30A}
\|T_\alpha\|_{2,2}
= \frac{1}{1+\alpha}\sqrt{\frac{\alpha}{h_0(\alpha)}}.
\end{equation}
Moreover,
\begin{align}
\|T_\alpha\|_{2,2}
&\sim \frac{1/\sqrt{h_0(\infty)}}{\sqrt{\alpha}}
\quad \text{as}\
\alpha\to\infty,                             \label{eq:31A}\\
\|T_\alpha - T_0\|_{2,2}
&\sim \sqrt{\alpha/h_0(\infty)}
\quad \text{as}\
\alpha\to 0,                                 \label{eq:32A}
\end{align}
where $h_0(\infty)= 1.445796\ldots$ is the smallest positive zero of $H_\infty(z) =
\sum_{k=0}^\infty\frac{(-z)^k}{(k!)^2}$,
and
\begin{equation}\label{eq:33A}
\|T_\alpha\|_{2,2}
= 1 - \frac{3}{4}\alpha + O(\alpha^2)
\ \text{as}\
\alpha\to 0.
\end{equation}
\end{corollary}
\noindent %
Formulas~\eqref{eq:31A}-\eqref{eq:32A}-\eqref{eq:33A} improve the general
formulas~\eqref{eq:3A}-\eqref{eq:4A}-\eqref{eq:5A}.
%
%G Halpha(alpha,N,z)=my(c=1/(1+alpha));1+sum(k=1,N,(-z)^k/k!/prod(j=1,k,j-c))
%G H0(alpha)=solve(z=0,z=2,Halpha(alpha,50,z))
%G \p 50
%
%G ploth(a=0.01,3,[1/(1+a)*sqrt(a/H0(a)), 1- 3/4*a])
%
\begin{proof}
The statement~\eqref{eq:30A} is an immediate consequence of~\eqref{eq:22A} in Theorem~\ref{th:5A}.
Claim~\eqref{eq:31A} is deduced from~\eqref{eq:30A} and a localization of the first zero for $H_\alpha$
in~\eqref{eq:21A} which is made possible via an application of Rouch\'{e} Theorem: we split this argument
into four lemmas, where for simplicity we have set $\eps := (1+\alpha)^{-1}$.
\begin{lemma}\label{lem:1A}
Let
\[
H(\eps,z) := \sum_{k=0}^\infty \frac{(-z)^k}{k!\prod_{j=1}^k(j-\eps)}.
\]
Then
\[
|H(\eps,z)- H(0,z)| \leq 5\eps e^{|z|}
\]
for every $z\in\C$, when $\eps \leq 1/8$.
\end{lemma}
\begin{proof}
In fact, for $\eps\leq 1/2$ and for a suitable $L$ that we will set later, we have
\begin{align*}
|H(\eps,z)-H(0,z)|
&\leq \sum_{k=1}^{L} \frac{|z|^k}{k!^2}\Big[\prod_{j=1}^k(1-\eps/j)^{-1}-1\Big]
     +\sum_{k=L+1}^\infty \frac{|z|^k}{k!}\Big[\frac{2^k}{(2k-1)!!}+\frac{1}{k!}\Big]\\
%
%&=    \sum_{k=1}^{L} \frac{|z|^k}{k!^2}\Big[\exp\Big(-\sum_{j=1}^k\log(1-\eps/j)\Big)-1\Big]
%     +\sum_{k=L+1}^\infty \frac{|z|^k}{k!}\Big[\frac{4^k k!}{(2k)!}+\frac{1}{k!}\Big]\\
%
&=    \sum_{k=1}^{L} \frac{|z|^k}{k!^2}\Big[\exp\Big(\sum_{j=1}^k \sum_{m=1}^\infty \frac{(\eps/j)^m}{m}\Big)-1\Big]
     +\!\!\sum_{k=L+1}^\infty \frac{|z|^k}{k!^2}\Big[\frac{4^k k!k!}{(2k)!}+1\Big].
\intertext{The term for $m=1$ is estimated using the inequality
$\sum_{j\leq k}1/j \leq \log(ek)$.
The remaining sum
$\sum_{m=2}^\infty \sum_{j=1}^k\frac{(\eps/j)^m}{m}$
is estimated by $\frac{\eps^2}{1-\eps}$.
The condition $\eps \leq 1/2$ is then used to prove that $\frac{\eps^2}{1-\eps}\leq \eps\log(ek)$.
Moreover, we notice that $\frac{4^k k!k!}{(2k)!} \leq 2k-1$ for $k\geq 2$. In this way we obtain the bound}
%
% NOTICE:
% 2^{2k} = 1 + 2k + \sum_{w=2}^{2k-2}\binom{2k}{2} + 2k + 1
% \leq 2 + 4k + (2k-3)\binom{2k}{k}
% and this is \leq (2k-1)\binom{2k}{k} as soon as 1+2k \leq \binom{2k}{k}, and this happens for every $k\geq 2$
%
&\leq \sum_{k=1}^{L} \frac{|z|^k}{k!^2}\Big[\exp\big(2\eps\log(ek)\big)-1\Big]
     +2\sum_{k=L+1}^\infty \frac{k|z|^k}{k!^2}\\
&\leq \sum_{k=1}^{L} \frac{|z|^k}{k!^2}\Big[\exp\big(2\eps\log(ek)\big)-1\Big]
     +\frac{2}{L!}e^{|z|}.
\intertext{We fix the value of $L$ to $\intpart{\exp(1/(2\eps)-1)}$; in this way $\exp(2\eps\log(ek))-1
\leq 4\eps\log(ek)$ inside the first sum because $e^y - 1 \leq 2y$ for $y\in[0,1]$. This yields}
&\leq 4\eps \sum_{k=1}^{L} \frac{|z|^k}{k!^2}\log(ek)
     +\frac{2e^{|z|}}{L!}.
\intertext{When $\eps \leq 1/8$ we have the inequality $L! = \intpart{\exp(1/(2\eps)-1)}! \geq
\exp(1/(2\eps)-1) \geq 2/\eps$, giving}
&\leq 4\eps \sum_{k=1}^{\infty} \frac{|z|^k}{k!^2}\log(ek)
     +\eps e^{|z|}.
\end{align*}
The claim now follows, because $\log(ek)/k! \leq 1$ for every $k\geq 1$.
\end{proof}
\begin{lemma}\label{lem:2A}
$H(0,z)$ has no zeros on the circle $|z|=3/2$.
\end{lemma}
\begin{proof}
In fact,
\[
\sum_{k=4}^\infty \frac{|z|^k}{k!^2}
\leq \frac{1}{4!}\sum_{k=4}^\infty \frac{|z|^k}{k!}
=    \frac{1}{4!}\Big[\exp(|z|)- \sum_{k=0}^3 \frac{|z|^k}{k!}\Big]
\leq 0.02
\quad
\text{when }
|z|=3/2.
\]
%G  my(w=3/2);(exp(w)-sum(k=0,3,w^k/k!))/4!
%G\\ 0.012257877930752700941752310838303159125
%
On the other hand, on the circle $|z|=3/2$, the polynomial $|\sum_{k=0}^3 (-z)^k/k!^2|$ attains its
minimum at $z=3/2$, with value $1/32 = 0.03125$;
%G NOTICE:
% A CHECK:
%G f(z)=sum(k=0,3,(-z)^k/k!^2);
%G ploth(u=0,2*Pi,abs(f(3/2*exp(I*u))))
%
% A PROOF: take the modulo square and write it as trigonometric polynomial:
% for
% \[
% |\sum_{k=0}^3\frac{(-3/2)^k e^{ik\theta}}{k!^2}|^2
% = \sum_{k,\ell=0}^3\frac{(-3/2)^{k+\ell} e^{i(k-\ell)\theta}}{k!^2}
% = \frac{3661}{1024} - \frac{1227}{256}\cos\theta + \frac{45}{32}\cos(2\theta) - \frac{3}{16}\cos(3\theta)
% \]
%G \\ if $F(\theta) = \sum_{k=0}^N a_k e^{2\pi i k\theta}$,
%G \\ then $|F(z)|^2 = \sum_{n=0}^N (\sum_{k,l: |k-l|=n}) \cos(2\pi n \theta)$
%G A(k)= (-3/2)^k/k!^2   \\ the coefficients
%G C(n)= sum(k=0,3,sum(l=0,3,if(abs(k-l)==n,A(k)*A(l)))) \\ the coefficient of f(z) when N=3
%G [C(0),C(1),C(2),C(3)]
%G \\ [3661/1024, -1227/256, 45/32, -3/16]
%
% Its derivative is
% \[
%  - \frac{1227}{256}\sin\theta + \frac{45}{16}\sin(2\theta) - \frac{9}{16}\sin(3\theta)
%  = sin(\theta)[-\frac{1083}{256} + \frac{45}{8}\cos\theta - \frac{9}{4}\cos^2\theta]
% \]
% and the term in bracket is not zero, thus extremal point are for $\theta=0,\pi$.
%
hence, for $|z|=3/2$ one has $|H(0,z)|\geq |\sum_{k=0}^3(-z)^k/k!^2| - |\sum_{k=4}^\infty (-z)^k/k!^2|
\geq 0.03125-0.02 > 0$.
\end{proof}
\begin{lemma}\label{lem:3A}
$H(0,z)$ has a unique complex zero in $|z|\leq 3/2$.
\end{lemma}
\begin{proof}
The polynomial $\sum_{k\leq 3} \frac{(-z)^k}{(k!)^2}= 1- z + \frac{z^2}{4} - \frac{z^3}{36}$ has a unique
zero in that disk.
%G
%G polroots(1-z+z^2/4-z^3/36)
%G\\ [1.4299934735128195631790897282155734827 + 0.E-38*I,
%G\\  3.7850032632435902184104551358922132586 - 3.2937350181684707292491988083245647728*I,
%G\\  3.7850032632435902184104551358922132586 + 3.2937350181684707292491988083245647728*I]~
%G abs(polroots(1-z+z^2/4-z^3/36))
%G\\ [1.4299934735128195631790897282155734827, 5.0174635098497610216774873448647030284, 5.0174635098497610216774873448647030284]~
%G
Moreover, during the proof of Lemma~\ref{lem:2A} we have shown that on the circle $|z|=3/2$
\[
\sum_{k=4}^\infty \frac{|z|^k}{k!^2}
\leq 0.02
<    0.03125
\leq \Big|\sum_{k=0}^3\frac{(-z)^k}{k!^2}\Big|.
\]
The claim follows by Rouch\'{e} Theorem.
\end{proof}
\noindent %
By Lemma~\ref{lem:1A} on the disc $|z|\leq 3/2$ we have that $|H(\eps,z)-H(0,z)|=O(\eps)$ and
Lemma~\ref{lem:2A} shows that the minimum for $|H(0,z)|$ on the circle $|z|= 3/2$ is positive. By Rouch\'{e}
Theorem, these facts prove that functions $H(\eps,\cdot)$ and $H(0,\cdot)$ have the same number of zeros
in that disk when $\eps$ is small enough, and hence a unique zero, by Lemma~\eqref{lem:3A}.
\begin{lemma}\label{lem:4A}
$|H'(\eps,z)| \geq 0.02$ for $|z|\leq 3/2$ and $\eps < 1/100$.
\end{lemma}
\begin{proof}
In fact,
\begin{align*}
|H'(\eps,z)|
&= \Big|\sum_{k=1}^{\infty} \frac{k (-z)^{k-1}}{k!\prod_{j=1}^k(j-\eps)}\Big|
\geq \frac{1}{1-\eps} - \Big|\sum_{k=2}^{\infty} \frac{k (-z)^{k-1}}{k!\prod_{j=1}^k(j-\eps)}\Big|\\
&\geq \frac{1}{1-\eps} - \sum_{k=2}^{\infty} \frac{k |z|^{k-1}}{k!\prod_{j=1}^k(j-\eps)}
 \geq \frac{1}{1-\eps} - \sum_{k=2}^{\infty} \frac{k |z|^{k-1}}{k!\prod_{j=1}^k(j-1/100)}\\
&\geq \frac{1}{1-\eps} - \sum_{k=2}^{\infty} \frac{k (3/2)^{k-1}}{k!\prod_{j=1}^k(j-1/100)}
\geq \frac{1}{1-\eps} - 0.98
\geq 0.02.
\qedhere
\end{align*}
% NOTICE:
% \sum_{k=N}^{\infty} \frac{k (3/2)^{k-1}}{k!\prod_{j=1}^k(j-1/100)}
% \leq \sum_{k=N}^{\infty} \frac{k (3/2)^{k-1}}{k!\prod_{j=1}^k j}
% \leq \sum_{k=N}^{\infty} \frac{k (3/2)^{k-1}}{k! k!}
% \leq \frac{1}{N!}\sum_{k=N}^{\infty} \frac{k (3/2)^{k-1}}{k!}
% \leq \frac{\exp(3/2)}{N!}
%
%G sum(k=2,10,k*(3./2)^(k-1)/k!/prod(j=1,k,j-1./100)) + exp(3/2)/10!
%G \\ = 0.97819072769421418067006835825477213252
%
\end{proof}
\noindent %
Now we can conclude. Assume $\eps$ small enough and let $z(\eps)$ be the unique zero of $H(\eps,z)$ in
the disk $|z|\leq 3/2$. The mean value Theorem and equalities $H(0,z(0))= 0 = H(\eps,z(\eps))$ give that
\[
(z(0)-z(\eps))H'(\eps,\eta)
= H(\eps,z(0)) - H(\eps,z(\eps))
= H(\eps,z(0)) - H(0,z(0))
\]
for some $\eta$ with $|\eta|\leq 3/2$ so that by Lemma~\ref{lem:4A}
\[
|z(\eps)-z(0)|
%=    \frac{|H(\eps,z(0)) - H(0,z(0))|}{|H'(\eps,\xi)|}
\leq 50 |H(\eps,z(0)) - H(0,z(0))|.
\]
This relation and Lemma~\ref{lem:1A} prove that $z(\eps) \to z(0)$ as $\eps \to 0$. This conclude the
proof of~\eqref{eq:31A}.
\medskip\\
Claim~\eqref{eq:32A} follows from~\eqref{eq:31A} and the identity $T_0 - T_\alpha = T^*_{1/\alpha}$.
\medskip\\
To prove~\eqref{eq:33A} we use directly~\eqref{eq:29A} in order to handle the fact that the term
$\prod_{j=1}^k(j-\frac{1}{1+\alpha})^{-1}$ contains the factor $(1-\frac{1}{1+\alpha})^{-1}$ which
diverges as $\alpha \to 0$.
Equation~\eqref{eq:29A} can be written as
\[
f(\alpha,\omega^2)
:= g(\alpha,\omega^2)
  - \sum_{k=3}^\infty \frac{(\frac{-\alpha}{(1+\alpha)^2})^{k-1}\omega^{2k}}{k!\prod_{j=2}^k(j-\frac{1}{1+\alpha})}
 = 0
\]
where
\[
g(\alpha,z)
:= 1 + \alpha - z + \frac{\alpha z^2}{2(1+\alpha)(1+2\alpha)}.
\]
When $\alpha\to 0$, the function $g(\alpha,z)$ has a unique zero, $z_0(\alpha)$ say, located in $|z|\leq
3/2$, and $z_0(\alpha)$ behaves as $1+3\alpha/2 + O(\alpha^2)$. Via Rouch\'{e} Theorem one proves that the
same happens to the full function $f(\alpha,z)$.
In fact,
\begin{align}
\big|f(\alpha,z) - g(\alpha,z)\big|
\leq \sum_{k=3}^\infty \frac{(\frac{\alpha}{1+\alpha})^{k-1}|z|^k}{k!\prod_{j=2}^k(j-1+j\alpha)}
\leq \sum_{k=3}^\infty \frac{\alpha^{k-1}|z|^k}{k!(k-1)!}
\leq \alpha^2|z|^3 e^{\alpha|z|}.                                                                \label{eq:34A}
\end{align}
In particular, the difference is $O(\alpha^2)$ as $\alpha\to 0$ and $z$ is bounded.
On the other hand, for $|z|=3/2$
\[
\big|g(\alpha,z)\big|
\geq |1 - z| + O(\alpha)
\geq 1/2 + O(\alpha)
\]
so that (by Rouch\'{e} Theorem) also $f(\alpha,z)$ has a unique zero in $|z|\leq 3/2$, when $\alpha$ is
small enough. %
Moreover, for $|z|\leq 3/2$
\begin{equation}\label{eq:35A}
\big|f'(\alpha,z)\big|
%= \Big|- 1 - \sum_{k=2}^\infty \frac{(\frac{-\alpha}{1+\alpha})^{k-1}k z^{k-1}}{k!\prod_{j=2}^k(j-1+j\alpha)}\Big|
\geq 1 - \sum_{k=2}^\infty \frac{(\frac{\alpha}{1+\alpha})^{k-1}k |z|^{k-1}}{k!\prod_{j=2}^k(j-1+j\alpha)}
\geq 1 - \sum_{k=2}^\infty \frac{\alpha^{k-1} |z|^{k-1}}{(k-1)!}
\geq 1 - \alpha|z|e^{\alpha|z|}
=    1 + O(\alpha).
\end{equation}
Let $z(\alpha)$ denote the zero of $f(\alpha,\cdot)$ in $|z|\leq 3/2$. Then $f(\alpha,z(\alpha)) = 0 =
g(\alpha,z_0(\alpha))$ and by the mean value Theorem there exists $\eta$ with $|\eta| \leq 3/2$ such that
\[
(z(\alpha)-z_0(\alpha))f'(\alpha,\eta)
= f(\alpha,z(\alpha)) - f(\alpha,z_0(\alpha))
= g(\alpha,z_0(\alpha)) - f(\alpha,z_0(\alpha))
\]
so that by~\eqref{eq:34A} and~\eqref{eq:35A}
\[
|z(\alpha) - z_0(\alpha)| \leq \frac{O(\alpha^2)}{1+O(\alpha)},
\]
i.e.,
\[
z(\alpha)
 = z_0(\alpha) + O(\alpha^2)
 = 1 +\frac{3}{2}\alpha + O(\alpha^2).
\]
The conclusion follows recalling that the norm is $1/\sqrt{z(\alpha)}$.
\end{proof}

%
%\bibliographystyle{amsplain}
%\bibliography{f:/books}
%\end{document}

\end{document}